\documentclass[english]{article}
\usepackage[T1]{fontenc}
\usepackage[latin9]{inputenc}
\usepackage{color}
\usepackage{amsthm}
\usepackage{amsmath}
\usepackage{amssymb}
\usepackage{esint}
\usepackage{latexsym}
\usepackage{graphicx}
\usepackage{mathscinet}
\usepackage{ulem}
\usepackage{mathtools} \mathtoolsset{showonlyrefs} 
\usepackage{enumitem}

\usepackage{csquotes}
\MakeOuterQuote{"}
\makeatletter
\theoremstyle{plain}
\newtheorem{thm}{\protect\theoremname}[section]
\theoremstyle{plain}

\newtheorem{cor}[thm]{\protect\corollaryname}
\newtheorem*{cor*}{\protect\corollaryname}
\newtheorem{lem}[thm]{\protect\lemmaname}
\newtheorem*{lem*}{\protect\lemmaname}

\newtheorem*{prop*}{\protect\propositionname}
\theoremstyle{remark}
\newtheorem{rem}{\protect\remarkname}
\usepackage{hyperref}
\usepackage{nameref}
\hypersetup{pdfpagemode=UseNone}

\DeclareMathOperator{\dist }{dist}

\DeclareMathOperator{\Proj}{Proj}

\makeatletter
\let\orgdescriptionlabel\descriptionlabel
\renewcommand*{\descriptionlabel}[1]{%
	\let\orglabel\label
	\let\label\@gobble
	\phantomsection
	\edef\@currentlabel{#1}%
	\let\label\orglabel
	\orgdescriptionlabel{#1}%
}
\makeatother

\makeatother

\usepackage{babel}
\providecommand{\corollaryname}{Corollary}
\providecommand{\lemmaname}{Lemma}
\providecommand{\remarkname}{Remark}
\providecommand{\theoremname}{Theorem}
\providecommand{\propositionname}{Proposition}

\def\RR{\mathbb{R}}

\def\HH{\mathcal{H}}

\def\eps{\varepsilon}

\def\Pp{\mathcal{P}}

\def\to{\rightarrow}

\def\tu{\tilde{u}}

\begin{document}
	
	\title{Regularity of solutions in semilinear elliptic theory}
	\date{}
	
	\author{E. Indrei, A. Minne, L. Nurbekyan}
	
	\def\signei{\bigskip\begin{center} {\sc Emanuel Indrei\par\vspace{3mm}Center for Nonlinear Analysis\\  
				Carnegie Mellon University\\
				Pittsburgh, PA 15213, USA\\
				email:} {\tt egi@cmu.edu }
		\end{center}}

		\def\signam{\bigskip\begin{center} {\sc Andreas Minne \par\vspace{3mm}
					Department of Mathematics\\
					KTH Royal Institute of Technology\\
					100 44 Stockholm, Sweden\\
					email:} {\tt minne@kth.se}
			\end{center}}
			
			\def\signln{\bigskip\begin{center} {\sc Levon Nurbekyan \par\vspace{3mm}
						CEMSE Division,\\
						King Abdullah University of Science and Technology (KAUST)\\
						Thuwal 23955-6900, Saudi Arabia\\
						email:} {\tt levon.nurbekyan@kaust.edu.sa}
				\end{center}}

				\maketitle
				
				\begin{abstract}
					We study the semilinear Poisson equation 
					\begin{equation} \label{pro}
					\Delta u = f(x, u) \hskip .2 in \text{in} \hskip .2 in B_1.
					\end{equation}
					Our main results provide conditions on $f$ which ensure that weak solutions of \eqref{pro} belong to $C^{1,1}(B_{1/2})$. In some configurations, the conditions are sharp.

				\end{abstract}

				\makeatletter
				\def\blfootnote{\gdef\@thefnmark{}\@footnotetext}
				\makeatother
				
				\section{Introduction}
				
				The semilinear Poisson equation \eqref{pro} encodes stationary states of the nonlinear heat, wave, and Schr\"odinger equation. In the case when $f$ is the Heaviside function in the $u$-variable, \eqref{pro} reduces to the classical obstacle problem. 
				 For an introduction to classical semilinear theory, see \cite{MR2722059,Caz}.        

It is well-known that weak solutions of \eqref{pro} belong to the usual Sobolev space $W^{2,p}(B_{1/2})$ for any $1\le p<\infty$ provided $f \in L^\infty$. Recent research activity has thus focused on identifying conditions on $f$ which ensure $W^{2,\infty}(B_{1/2})$ regularity of $u$.			
								
\subsection{The classical theory}
				  
There are simple examples which illustrate that continuity of $f=f(x)$ does not necessarily imply that $u$ has bounded second derivatives: for $p \in (0,1)$ and $x \in \mathbb{R}^2$ such that $|x|<1$, the function $$u(x)=x_1x_2(-\log|x|)^p$$ has a continuous Laplacian but is not in $C^{1,1}$ \cite{Sha15}. However, if $f$ is H\"older continuous, then it is well-known that $u \in C^{2,\alpha}$; if $f$ is Dini continuous, then $u \in C^2$ \cite{gilbarg:01, MR1713596}. The sharp condition which guarantees bounded second derivatives of $u$ is the $C^{1,1}$ regularity of $f * N$ where $N$ is the Newtonian potential and $*$ denotes convolution; this requirement is strictly weaker than Dini continuity of $f$. 

In the general case, the state-of-the-art is a theorem of Shahgholian \cite{S03} which states that $u \in C^{1,1}$ whenever $f=f(x,u)$ is Lipschitz in $x$, uniformly in $u$, and $\partial_u f \ge -C$ weakly for some $C \in \mathbb{R}$. In some configurations this illustrates regularity for continuous functions $f=f(u)$ which are strictly below the classical Dini-threshold in the $u$-variable, e.g. the odd reflection of $$f(t)=-\frac{1}{\log(t)}$$ about the origin. Shahgholian's theorem is proved via the celebrated Alt-Caffarelli-Friedman (ACF) monotonicity formula and it seems difficult to weaken the assumptions by this method. On the other hand, Koch and Nadirashvili \cite{KN} recently constructed an example which illustrates that the continuity of $f$ is not sufficient to deduce that weak solutions of $\Delta u = f(u)$ are in $C^{1,1}$.   

We say $f=f(x,u)$ satisfies assumption \textbf{A} provided that $f$ is Dini continuous in $u$, uniformly in $x$, and has a $C^{1,1}$ Newtonian potential in $x$, uniformly in $u$ (see \S 3). One of our main results is the following statement.   			


\begin{thm}\label{contcase}
					Suppose $f$ satisfies assumption \textbf{A}.
					Then any solution of \eqref{pro} is $C^{1,1}$ in $B_{1/2}$. 
				\end{thm}

Our assumption includes functions which fail to satisfy both conditions in Shahgholian's theorem, e.g. $$f(x_1,x_2,t)=\frac{x_1}{\log(|x_2|)(-\log |t|)^p},$$ for $p>1$, $x=(x_1,x_2) \in B_1$ and $t\in (-1,1)$. The Newtonian potential assumption in the $x$-variable is essentially sharp whereas the condition in the $t$-variable is in general not comparable with Shahgholian's assumption.  

The proof of Theorem \ref{contcase} does not invoke monotonicity formulas and is self-contained.  We consider the $L^2$ projection of $D^2u$ on the space of Hessians generated by second order homogeneous harmonic polynomials on balls with radius $r>0$ and show that the projections stay uniformly bounded as $r \rightarrow 0^+$. Although this approach has proven effective in dealing with a variety of free boundary problems \cite{ALS13, FS, IM15, IM2}, Theorem \ref{contcase} illustrates that it is also useful in extending and refining the classical elliptic theory.             
 
\subsection{Singular case: the free boundary theory}

In \S 4 we study the PDE \eqref{pro} for functions $f=f(x,u)$ which are discontinuous in the $u$-variable at the origin. 

If the discontinuity of $f$ is a jump discontinuity, \eqref{pro} has the structure						
\begin{equation}\label{jump}
f(x,u)=g_1(x,u)\chi_{\{u>0\}}+g_2(x,u)\chi_{\{u<0\}},
\end{equation}
where $g_1,g_2$ are continuous functions such that
\begin{equation}
g_1(x,0)\neq g_2(x,0),\quad \forall x \in B_1,
\end{equation}
and $\chi_\Omega$ defines the indicator function of the set $\Omega$.
									
									Our aim is to find the most general class of coefficients $g_i$ which generate interior $C^{1,1}$ regularity. 
									
									The classical obstacle problem is obtained by letting $g_1=1,g_2=0$, and it is well-known that solutions have second derivatives in $L^\infty$ \cite{MR2962060}. Nevertheless, by selecting $g_1=-1,g_2=0$, one obtains the so-called unstable obstacle problem. Elliptic theory and the Sobolev embedding theorem imply that any weak solution belongs to $C^{1,\alpha}$ for any $0<\alpha<1$. It turns out that this is the best one can hope for: there exists a solution which fails to be in $C^{1,1}$ \cite{MR2289547}. Hence, if there is a jump at the origin, $C^{1,1}$ regularity can hold only if the jump is positive and this gives rise to:
									
									\vskip .3 cm
									\textbf{Assumption B.} $g_1(x,0)-g_2(x,0) \geq \sigma_0,\ x \in B_1$ for some $\sigma_0>0$.
									\vskip .3 cm
									
									The free boundary $\Gamma=\partial \{u \neq 0\}$ consists of two parts: $\Gamma^0=\Gamma \cap \{\nabla u = 0\}$ and $\Gamma^1=\Gamma \cap \{\nabla u \neq 0\}$. The main difficulty in proving $C^{1,1}$ regularity is the analysis of points where the gradient of the function vanishes. In this direction we establish the following result. 
									
									\begin{thm}\label{AB}
 Suppose $g_1,g_2$ satisfy \textbf{A} and \textbf{B}. Then if $u$ is a solution of \eqref{pro}, $\|u\|_{C^{1,1}(K)}<\infty$ for any $K \Subset B_{1/2}(0)\setminus \overline{\Gamma^1}$.
\end{thm}

									At points where the gradient does not vanish, the implicit function theorem yields that the free boundary is locally a $C^{1,\alpha}$ graph for any $0<\alpha<1$. The solution $u$ changes sign across the free boundary, hence it locally solves the equation $\Delta u =g_1(x,u)$ on the side where it is positive and $\Delta u =g_2(x,u)$ on the side where it is negative. If the coefficients $g_i$ are regular enough to provide $C^{1,1}$ solutions up to the boundary -- this is encoded in assumption \textbf{C}, see \S 4 -- then we obtain full $C^{1,1}$ regularity. 
								
								\begin{thm}\label{ABC}
	Suppose $g_1,g_2$ satisfy \textbf{A}, \textbf{B} and \textbf{C}. Let $u$ be a solution of \eqref{pro} and $0\in\Gamma^0$. Then $u \in C^{1,1}(B_{\rho_0}(0))$, for some $\rho_0>0$.
\end{thm}

									Equation \eqref{pro} with right-hand side of the form \eqref{jump} is a generalization of the well-studied two-phase membrane problem, where $g_i(x,u)=\lambda_i(x),\ i=1,2$. The $C^{1,1}$ regularity in the case when $\lambda_1\ge 0$, $\lambda_2 \le 0$ are two constants satisfying \textbf{B} was obtained by Uraltseva \cite{MR1906034} via the ACF monotonicity formula. Moreover, Shahgholian proved this result for Lipschitz coefficients which satisfy \textbf{B} \cite[Example 2]{S03}. If the coefficients are H\"{o}lder continuous, the ACF method does not directly apply and under the stronger assumption that $\inf \lambda_1>0$ and $\inf -\lambda_2>0$, Edquist, Lindgren, Shahgholian \cite{LSE09} obtained the $C^{1,1}$ regularity via an analysis of blow-up limits and a classification of global solutions (see also \cite[Remark 1.3]{LSE09}). Theorem \ref{ABC} improves and extends this result.   
									
									The difficulty in the case when $g_i$ depend also on $u$ is that if $v:=u+L$ for some linear function $L$, then $v$ is no longer a solution to the same equation, so one has to get around the lack of linear invariance. Our technique exploits that linear perturbations do not affect certain $L^2$ projections.   
									
									The proof of Theorem \ref{ABC} does not rely on classical monotonicity formulas or classification of global solutions. Rather, our method is based on an identity which provides monotonicity in $r$ of the square of the $L^2$ norm of the projection of $u$ onto the space of second order homogeneous harmonic polynomials on the sphere of radius $r$.

Theorems \ref{AB} \& \ref{ABC} deal with the case when $f$ has a jump discontinuity.  If $f$ has a removable discontinuity, \eqref{pro} has the structure
\begin{equation}\label{pro: no-sign}
\Delta u=g(x,u) \chi_{u\neq 0}.
\end{equation}
In this case, one may merge some observations in the proofs of the previous results with the method in \cite{ALS13} and prove the following theorem.   

\begin{thm}\label{thm: no-sign} If $g$ satisfies assumption \textbf{A}, then every solution of \eqref{pro: no-sign} is in $C^{1,1}(B_{1/2})$.
\end{thm}

Theorems \ref{contcase} - \ref{thm: no-sign} provide a comprehensive theory for the general semilinear Poisson equation where the free boundary theory is encoded in the regularity assumption of $f$ in the $u$-variable.

				\section{Technical tools} 
				
				Throughout the text, the right-hand side of \eqref{pro} is assumed to be bounded. Moreover, $\mathcal{P}_2$ denotes the space of second order homogeneous harmonic polynomials. A useful elementary fact is that all norms on $\mathcal{P}_2$ are equivalent.
				\begin{lem}\label{lem: normsequivalent}
					The space $\Pp_2$ is a finite dimensional linear space. Consequently, all norms on $\Pp_2$ are equivalent.
				\end{lem}
				
				\noindent For $u \in W^{2,2}(B_1)$, $y \in B_1$ and $r \in (0, \dist(y, \partial B_1))$, $\Pi_y(u,r)$ is defined to be the $L^2$ projection operator on $\mathcal{P}_2$ given by 
				
				$$\inf_{h \in \mathcal{P}_2} \int_{B_1} \Big | D^2\frac{u(rx+y)}{r^2}-D^2h \Big|^2 dx=\int_{B_1} \Big |D^2\frac{u(rx+y)}{r^2}-D^2\Pi_y(u,r) \Big|^2 dx.$$ Calderon-Zygmund theory yields the following useful inequality for re-scalings of weak solutions of \eqref{pro}.
				
					\begin{lem} \label{w2p} 
						Let $u$ solve \eqref{pro}, $y \in B_{1/2}$, and $r\le 1/4$. Then for 
						\[\tu_r(x)=\frac{u(rx+y)-rx\cdot\nabla u(y)-u(y)}{r^2}\]
						it follows that for $1\le p<\infty$ and $0<\alpha<1$, 
						$$\|\tu_r - \Pi_y(u,r)\|_{W^{2,p}(B_1)} \le C(n,\|f\|_{L^\infty(B_1\times \mathbb{R})},\|u\|_{L^\infty(B_1)}, p),$$
						and
						$$\|\tu_r - \Pi_y(u,r)\|_{C^{1,\alpha}(B_1)} \le C(n,\|f\|_{L^\infty(B_1\times \mathbb{R})},\|u\|_{L^\infty(B_1)}, \alpha).$$
					\end{lem}

					\begin{proof}
						By Calderon-Zygmund theory (e.g. \cite[Theorem 2.2]{ALS13}), $$\|D^2u\|_{BMO(B_{1/2})} \le C;$$ in particular, 
						$$\int_{B_{3/2}} |D^2 \tilde{u}_r - \overline{D^2 \tilde{u}_r}|^2 \le C,$$ where $\overline{D^2 \tilde{u}_r}$ is the average of $D^2 \tilde{u}_r$ on $B_{3/2}.$ Now let \[a=a(f, r, y)=\fint_{B_{3/2}} f(rx+y, u(rx+y))\,dx\] and note that this quantity is uniformly controlled by $\|f\|_{L^\infty(B_1\times \mathbb{R})}$; this fact, and the definition of $\Pi$ yields (note: $\text{trace}(\overline{D^2 u}-\frac{a}{n}Id)=0$), 
						\begin{equation} \label{bmo}
						\int_{B_{3/2}} |D^2 (\tilde{u}_r - \Pi_0(\tilde{u}_r, 3/2))|^2 \le \int_{B_{3/2}} |D^2 \tilde{u}_r - (\overline{D^2 u}-\frac{a}{n}Id)|^2 \le C_1.
						\end{equation}
						Two applications of Poincar\'e's inequality together with the above estimate implies 
						$$\|\tilde{u}_r -\Pi_y(u, r)- \overline{\nabla \tilde{u}_r} \cdot x- \overline{\tilde{u}_r}\|_{W^{2,2}(B_{3/2})} \le C_2,$$
						where the averages are taken over $B_{3/2}$. Elliptic theory (e.g. \cite[Theorem 9.1]{gilbarg:01}) yields that for any $1\le p<\infty$,
						$$\|\tilde{u}_r -\Pi_y(u, r)- \overline{\nabla \tilde{u}_r} \cdot x- \overline{\tilde{u}_r}\|_{W^{2,p}(B_{3/2})} \le C_3.$$
						Let $\phi:=\tilde{u}_r-\overline{\nabla \tilde{u}_r} \cdot x- \overline{\tilde{u}_r}$. We have that $\phi(0)=-\overline{\tilde{u}_r}$ and $\nabla \phi(0)=-\overline{\nabla \tilde{u}_r}$; however, by the  Sobolev embedding theorem, $\phi$ is $C^{1,\alpha}$ and thus $$|\phi(0)|+|\nabla \phi(0)| \le C_4$$ completing the proof of the $W^{2,p}$ estimate. The $C^{1,\alpha}$ estimate likewise follows from the Sobolev embedding theorem.   
					\end{proof}
				
				\noindent Our analysis requires several additional simple technical lemmas involving the projection operator. 
				
				\begin{lem}\label{lma: keyineq}
					For any $u \in W^{2,2}(B_1)$ and $s\in [1/2,1]$,
					\begin{equation}\label{keyineq}
					\|\Pi_0(u,s)-\Pi_0(u,1)\|_{L^2(B_1)} \leq C \|\Delta u\|_{L^2(B_1)},
					\end{equation}
					and
					\begin{equation}\label{keyineq1}
					\|\Pi_0(u,s)-\Pi_0(u,1)\|_{L^\infty(B_1)} \leq C \|\Delta u\|_{L^2(B_1)},
					\end{equation}
					for some constant $C=C(n)$.
				\end{lem}
				\begin{proof}
					Let $f=\Delta u$ and $v$ be the Newtonian potential of $f$, i.e.
					$$v(x)=\frac{1}{n(n-2)\omega_n}\int\limits_{\RR ^n} \frac{f(y)\chi_{B_1}(y)}{|x-y|^{n-2}}dx,$$
					where $\omega_n$ is the volume of the unit ball in $\mathbb{R}^n$.
					Since $u-v$ is harmonic, $$\Pi_0(u-v,s)=\Pi_0(u-v,1);$$ therefore 
					$$\Pi_0(u,s)-\Pi_0(u,1)=\Pi_0(v,s)-\Pi_0(v,1).$$ Invoking bounds on the projection (e.g. \cite[Lemma 3.2]{ALS13}) and Calderon-Zygmund theory (e.g. \cite[Theorem 2.2]{ALS13}), it follows that
					\begin{align*}
					\|\Pi_0(u,s)-\Pi_0(u,1)\|_{L^2(B_1)}&=\|\Pi_0(v,s)-\Pi_0(v,1)\|_{L^2(B_1)}\\ 
					&\leq C \|\Delta v\|_{L^2(B_1)}=C \|\Delta u\|_{L^2(B_1)}.
					\end{align*}
					The $L^{\infty}$ bound follows from the equivalence of the norms in the space $\Pp_2$.
				\end{proof}
				
				\begin{lem} \label{log}
					Let $u$ solve \eqref{pro}. Then for all $0< r \le 1/4,\ s\in [1/2,1]$ and $y \in B_{1/2}$, 
					$$\sup_{B_1} |\Pi_y(u,rs)-\Pi_y(u,r)| \le C,$$
					and
					\[\sup\limits_{B_1}|\Pi_y(u,r)|\leq C\log(1/r),\]
					for some constant $C=C(n,\|f\|_{L^\infty(B_1\times \mathbb{R})},\|u\|_{L^\infty(B_1)})$.
				\end{lem}
				
				\begin{proof}
					Note that
					\[\Pi_y(u,rs)-\Pi_y(u,r)=\Pi_0(\tu_r,s)-\Pi_0(\tu_r,1),\]
					where
					\[\tu_r(x)=\frac{u(rx+y)-rx\cdot\nabla u(y)-u(y)}{r^2}\]
					as before. From Lemma \ref{lma: keyineq} we have that
					\[\|\Pi_0(\tu_r,s)-\Pi_0(\tu_r,1)\|_{L^{\infty}(B_1\times \mathbb{R})}\leq C \|\Delta \tu_r\|_{L^2(B_1)}\leq C \|f\|_{L^{\infty}(B_1)}.\]
					As for the second inequality in the statement of the lemma let $r_0=1/4$ and $s\in[1/2,1]$. Then we have that
					\begin{align}
					\sup\limits_{B_1}|\Pi_y(u,sr_0/2^j)| &\leq  \sup \limits_{B_1}|\Pi_y(u,sr_0/2^j)-\Pi_y(u,r_0/2^j)|
					\\
					&+\sum\limits_{k=0}^{j-1}\sup\limits_{B_1}|\Pi_y(u,r/2^{k+1})-\Pi_y(u,r/2^k)\\
					&+\sup\limits_{B_1}|\Pi_y(u,r_0)|\leq C j\leq C \log \left(\frac{2^j}{sr_0}\right),
					\end{align}
					for all $j\geq 1$.
				\end{proof}
				
				\noindent The previous tools imply a growth estimate on weak solutions solution of \eqref{pro}. 
				\begin{lem} \label{gro}
					Let $u$ solve \eqref{pro}. Then for $y \in B_{1/2}$ and $r>0$ small enough,
					$$
					\sup_{B_r(y)} |u(x)-u(y)-(x-y)\nabla u(y)| \le Cr^2 \log(1/r). 
					$$
				\end{lem}
				\begin{proof}
					Let
					\[\tilde{u}_r=\frac{u(rx+y)-rx \cdot \nabla u(y)-u(y)}{r^2}.
					\]
					The assertion of the Lemma is equivalent to the estimate \[\|\tu_r\|_{L^{\infty}(B_1)}\leq C \log(1/r),\] for $r$ small enough.
					Lemma \ref{log} and the $C^{1,\alpha}$ estimates of Lemma \ref{w2p} imply
					\begin{align}
					\|\tu_r\|_{L^{\infty}(B_1)} &\leq \|\tu_r-\Pi_y(u,r)\|_{L^{\infty}(B_1)}+\|\Pi_y(u,r)\|_{L^{\infty}(B_1)}\\
					&\leq C+C\log(1/r)\leq C\log(1/r),
					\end{align}
					provided $r$ is small enough.
				\end{proof}
				\noindent  Next lemma relates the boundedness of the projection operator and the boundedness of second derivatives of weak solutions of \eqref{pro}.   
				\begin{lem} \label{c11}
					Let $u$ be a solution to \eqref{pro}. If for each $y \in B_{1/2}$ there is a sequence $r_j(y) \rightarrow 0^+$ as $j \rightarrow \infty$ such that $$M:=\sup_{y\in B_{1/2}}\sup_{j \in \mathbb{N}} \|D^2\Pi_y(u,r_j(y))\|_{L^\infty(B_{1/2})} < \infty,$$ then 
					$$|D^2 u| \le C \hskip .2in \text{a.e. in} \hskip .2 in B_{1/2},$$ for some constant $C=C(M,n,\|f\|_{L^\infty(B_1 \times \mathbb{R})},\|u\|_{L^{\infty}(B_1)})>0$. 
				\end{lem}
				
				\begin{proof}
					Let $y \in B_{1/2}$ be a Lebesgue point for $D^2 u$ and $r_j=r_j(y) \rightarrow 0^+$ as $j \rightarrow \infty$. Then by utilizing Lemma \ref{w2p}, 
					\begin{align*}
					|D^2 u(y)| &= \lim_{j\rightarrow \infty} \fint_{B_{r_j}(y)} |D^2u(z)|dz\\
					&\le \limsup_{j\rightarrow \infty} \fint_{B_{r_j}(y)} |D^2u(z)-D^2\Pi_y(u,r_j)|dz + M\\
					&\le C. 
					\end{align*}
					Since a.e. $z \in B_{1/2}$ is a Lebesgue point for $D^2 u$, the proof is complete.   
				\end{proof}
				
				Next, we introduce another projection that we need for our analysis. Define $Q_y(u,r)$ to be the minimizer of
				\[\inf\limits_{q \in \mathcal{P}_2} \int\limits_{\partial B_1} \left|\frac{u(rx+y)}{r^2}-q(x)\right|^2 d\mathcal{H}^{n-1}.
				\]
				The following lemma records the basic properties enjoyed by this projection, cf. \cite[Lemma 3.2]{ALS13}.
				\begin{lem}\label{Q}
					\begin{itemize}
						\item[i.] $Q_y(\cdot,r)$ is linear;
						\item[ii.] if $u$ is harmonic $Q_y(u,s)=Q_y(u,r)$ for all $s<r$;
						\item[iii.] if $u$ is a linear function then $Q_y(u,r)=0$;
						\item[iv.] if $u$ is a second order homogeneous polynomial then $Q_y(u,r)=u$;
						\item[v.] $\|Q_0(u,s)-Q_0(u,1)\|_{L^2(\partial B_1)}\leq C_s \|\Delta u\|_{L^2(B_1)}$, for $0<s<1$;
						\item[vi.] $\|Q_0(u,1)\|_{L^2(\partial B_1)} \leq \|u\|_{L^2(\partial B_1)}$.
					\end{itemize}
				\end{lem}
				\begin{proof}
					\begin{itemize}
						\item[i.] This is evident.
						\item[ii.] It suffices to prove $Q_y(u,r)=Q_y(u,1)$ for $r<1$. Let $$\sigma_2=\frac{Q_y(u,1)}{\|Q_y(u,1)\|_{L^2(\partial B_1)}}$$ and for $i \neq 2$, let $\sigma_i $ be an $i^{th}$ degree harmonic polynomial. Then there exist coefficients $a_i$ such that
						\[u(x+y)=\sum\limits_{i=0}^{\infty} a_i \sigma_i(x),\quad x\in \partial B_1;
						\]
						in particular, $a_2=\|Q_y(u,1)\|$.
						Let
						\[v(x)=\sum\limits_{i=0}^{\infty} a_i \sigma_i(x),\quad x\in B_1.
						\]
						Then $v$ is a harmonic and $u(x+y)=v(x)$ for $x \in \partial B_1$. Hence, we have that $u(x+y)=v(x)$ for $x \in B_1$ and in particular
						\[u(x+y)=\sum\limits_{i=0}^{\infty} a_i \sigma_i(x),\quad x\in B_1.
						\]
						Therefore
						\[\frac{u(rx+y)}{r^2}=\sum\limits_{i=0}^{\infty} a_i \frac{\sigma_i(rx)}{r^2}=\sum\limits_{i=0}^{\infty} a_i r^{i-2} \sigma_i(x),\quad x\in B_1,
						\]
						so $Q_y(u,r)=a_2\sigma_2(x)=Q_y(u,1)$.
						\item[iii. \& iv.] These are evident.
						\item[v.] Similar to Lemma \ref{keyineq}.
						\item[vi.] This follows from the fact that $Q_0(u,1)$ is the $L^2$ projection of $u$.
					\end{itemize}
				\end{proof}
				
				Next we prove some technical results for $Q_y(u,r)$ and establish a precise connection between $\Pi_y(u,r)$ and $Q_y(u,r)$ by showing that the difference is uniformly bounded in $r$.

				\begin{lem}\label{Qderivative}
					For $u \in W^{2,p}(B_1(y))$ with $p$ large enough and $r \in (0,1]$,
					\begin{equation}\label{eq: Qderivative}
					\frac{d}{dr}Q_y(u,r)=\frac{1}{r}Q_0(x\cdot \nabla u(x+y)-2u(x+y),r).
					\end{equation}
				\end{lem}
				\begin{proof}
					Firstly, \[Q_y(u,r)=Q_0\left(\frac{u(rx+y)}{r^2},1 \right).\]
					Since $u$ is $C^{1,\alpha}$ if $p$ large enough and $Q$ is linear bounded operator, it follows that
					\begin{align}
					\frac{d}{dr}Q_y(u,r)&=Q_0\left(\frac{d}{dr}\frac{u(rx+y)}{r^2},1\right)=Q_0\left(\frac{rx\cdot \nabla u(rx+y)-2u(rx+y)}{r^3},1\right)\\
					&=\frac{1}{r}Q_0(x\cdot \nabla u(x+y)-2u(x+y),r).
					\end{align}
				\end{proof}
				\begin{lem}\label{lem: intbyparts}
					Let $u \in W^{2,p}(B_1(y))$ with $p$ large enough and $q \in \Pp_2$. Then
					\begin{equation}\label{intbyparts}
					\int \limits_{B_1} q(x) \Delta u(x+y) dx = \int\limits_{\partial B_1} q(x) \left(x\cdot \nabla u(x+y)-2u(x+y)\right)d \mathcal{H}^{n-1}.
					\end{equation}
				\end{lem}
				\begin{proof}
					Integration by parts implies
					\[\int \limits_{B_1} q(x) \Delta u(x+y) dx = \int \limits_{B_1} \Delta q(x) u(x+y) dx +\int\limits_{\partial B_1} q(x) \frac{\partial u(x+y)}{\partial n}-u(x+y) \frac{\partial q(x)}{\partial n}d \mathcal{H}^{n-1}.
					\]
					By taking into account that $q$ is a second order homogeneous polynomial it follows that 
					\[\frac{\partial q(x)}{\partial n}=2q(x),\quad x \in \partial B_1.
					\]
					Moreover,
					\[\frac{\partial u(x+y)}{\partial n}=x\cdot \nabla u(x+y),\quad x \in \partial B_1.
					\]
					Combining these equations yields \eqref{intbyparts}.
				\end{proof}
				
				\begin{lem}\label{keylemma}
					Let $u \in W^{2,p}(B_1(y))$ with $p$ large enough and $0<r\leq 1$. Then for every $q\in \Pp_2$, 
					\begin{equation}\label{lemma1}
					\int\limits_{\partial B_1} q(x) \frac{d}{dr} Q_y(u,r)(x) d\mathcal{H}^{n-1}=\frac{1}{r}\int \limits_{B_1} q(x) \Delta u(rx+y) dx.
					\end{equation}
				\end{lem}
				\begin{proof}Let $\tu_r(x)=u(rx+y)/r^2$.
					From Lemmas \ref{Qderivative} and \ref{lem: intbyparts} we obtain
					\begin{align*}
					\int\limits_{\partial B_1} q(x) \frac{d}{dr} Q_y(u,r)(x) d\mathcal{H}^{n-1}&=\frac{1}{r}\int\limits_{\partial B_1} q(x) Q_0\left(\frac{rx\cdot \nabla u(rx+y)-2u(rx+y)}{r^2},1\right) d\mathcal{H}^{n-1}\\
					&=\frac{1}{r}\int\limits_{\partial B_1} q(x) Q_0\left(x\cdot \nabla \tu_r(x)-2\tu_r(x),1\right) d\mathcal{H}^{n-1}\\
					&=\frac{1}{r}\int\limits_{\partial B_1} q(x) \left(x\cdot \nabla \tu_r(x)-2\tu_r(x)\right) d\mathcal{H}^{n-1}\\
					&=\frac{1}{r}\int\limits_{B_1} q(x) \Delta \tu_r(x) dx=\frac{1}{r}\int \limits_{B_1} q(x) \Delta u(rx+y) dx.
					\end{align*}
				\end{proof}
				
				\begin{lem}\label{normderivative}
					For $u \in W^{2,p}(B_1(y))$ with $p$ large enough and $0<r\leq 1$, 
					\begin{equation}\label{eq: normderivative}
					\frac{d}{dr} \int\limits_{\partial B_1} Q_y^2(u,r) d \HH^{n-1}=\frac{2}{r}\int\limits_{B_1} Q_y(u,r) \Delta u(rx+y) d x.
					\end{equation}
				\end{lem}
				\begin{proof}
					By Lemmas \ref{Qderivative}, \ref{keylemma} we get
					\begin{align}
					\frac{d}{dr} \int\limits_{\partial B_1} Q_y^2(u,r) d \HH^{n-1}&=2 \int\limits_{\partial B_1} Q_y(u,r) \frac{d}{dr} Q_y(u,r) d \HH^{n-1}\\
					&=\frac{2}{r}\int\limits_{B_1} Q_y(u,r) \Delta u(rx+y) d x.
					\end{align}
				\end{proof}

				\begin{lem}\label{u-v} Let $u$ be a solution of \eqref{pro} and $y \in B_{1/2}$. For $0<r<1/2$ consider \[u_r(x):=\frac{u(rx+y)-rx\cdot \nabla u(y)-u(y)}{r^2}-\Pi_y(u,r),\]
					
					\[v_r(x):=\frac{u(rx+y)-rx\cdot \nabla u(y)-u(y)}{r^2}-Q_y(u,r).
					\]
					Then
					\begin{itemize}
						\item[i.] $u_r-v_r$ is bounded in $C^{\infty}$, uniformly in $r$;
						\item[ii.] the family $\{v_r\}$ is bounded in $C^{1,\alpha}(B_1)\cap W^{2,p}(B_1)$, for every $0<\alpha<1$ and $p>1$.
					\end{itemize}
				\end{lem}
				\begin{proof}
					i. For each $r$, the difference $u_r-v_r=Q_y(u,r)-\Pi_y(u,r)$ is a second order harmonic polynomial. Therefore, it suffices to show that $L^{\infty}$ norm of that difference admits a bound independent of $r$. Note that
					\begin{align*}
					u_r-v_r&=Q_y(u,r)-\Pi_y(u,r)\\
					&=Q_0\left(\frac{u(rx+y)-rx\cdot \nabla u(y)-u(y)}{r^2}-\Pi_y(u,r),1\right)=Q_0(u_r,1).
					\end{align*}
					Hence,
					\begin{align*}
					\sup\limits_{r}	\sup\limits_{B_1} |Q_0(u_r,1)|\leq C \sup\limits_{r}\sup\limits_{B_1} |u_r|<\infty.
					\end{align*}
					ii. Lemma \ref{w2p} implies that $\{u_r\}_{r>0}$ is bounded in $C^{1,\alpha}(B_1)\cap W^{2,p}(B_1)$ for every $\alpha<1$ and $p>1$. Hence, the result follows from i.
				\end{proof}
				
				\section{$C^{1,1}$ regularity: general case}

			 	In this section we utilize the previous technical tools and prove $C^{1,1}$ regularity provided that $f=f(x,t)$ satisfies assumption \textbf{A}: 
				
\textbf{Assumption $\mathbf{A}$.} \begin{description}[leftmargin=0pt]
	\item[(i)\label{alt(i)}]  \begin{equation}
	|f(x,t_2)-f(x,t_1)|\leq h(x)\omega(|t_2-t_1|),
	\end{equation}
	where $h \in L^\infty(B_1)$ and
	\[\int\limits_{0}^{\epsilon} \frac{\omega(t)}{t}dt<\infty,
	\]
	for some $\epsilon>0$;\\
	\item[(ii)\label{alt(ii)}] The Newtonian potential of $x \mapsto f(x,t)$ is $C^{1,1}$ locally uniformly in $t$: for $v_t:=f(\cdot,t) * N$ where $N$ is the Newtonian potential, 
	$$
	\sup_{a\leq t \leq b} \|D^2 v_t\|_{L^\infty(B_1)} < \infty,\quad \text{for all}\quad a,b \in \RR.
	$$ 
\end{description}

				\begin{proof}[\bf{Proof of Theorem \ref{contcase}}]
					Let $y \in B_{1/2}$ and $v=v_{u(y)}=f(x,u(y)) * N$. 
					Note that if
					\[u_r(x)=\frac{u(rx+y)-rx \cdot \nabla u(y)-u(y)}{r^2}-\Pi_y(u,r),
					\]
					then
					\[\Pi_y(u,r/2)-\Pi_y(u,r)=\Pi_y(u_r,1/2)-\Pi_y(u_r,1)=\Pi_y(u_r,1/2).
					\]
					Using this identity, Lemma \ref{lma: keyineq}, and Lemma \ref{gro}
					\begin{align*}
					&\phantom{{}={}}\|\Pi_y(u,r/2)-\Pi_y(u,r)-\Pi_y(v,r/2)+\Pi_y(v,r)\|_{L^\infty(B_1)}\\
					&=\|\Pi_y(u_r,1/2)-\Pi_y(v_r,1/2)-\Pi_y(u_r,1)+\Pi_y(v_r,1)\|_{L^\infty(B_1)}\\
					&= \|\Pi_y(u_r-v_r,1/2)-\Pi_y(u_r-v_r,1)\|_{L^\infty(B_1)}\\
					&\leq C \|\Delta u_r-\Delta v_r\|_{L^2(B_1)}\\
					&=\|f(rx+y,u(rx+y))-f(rx+y,u(y))\|_{L^2(B_1)}\\
					&\leq C \omega\left(\sup \limits_{B_r(y)}|u(x)-u(y)|\right) \leq C \omega\left(c(r+r^2 \log \frac{1}{r})\right) \le C\omega\left(cr\right) ,
					\end{align*}
					for $r>0$ sufficiently small ($|\nabla u(y)|$ is controlled by $\|u\|_{W^{2,p}(B_1)}$). Hence, for $r_0>0$ small enough and $y \in B_{1/2}$ we have
					\begin{align*}
					&\phantom{{}={}}\|\Pi_y(u,r_0/2^j)-\Pi_y(u,r_0)\|_{L^\infty(B_1)}\\
					&\leq \bigg\|\sum\limits_{k=1}^{j} \Pi_y(v,r_0/2^k)-\Pi_y(v,r_0/2^{k-1})\bigg\|_{L^\infty(B_1)}\\
					&+ \sum\limits_{k=1}^{j} \bigg \|\Pi_y(u,r_0/2^k)-\Pi_y(u,r_0/2^{k-1})-\Pi_y(v,r_0/2^k)+\Pi_y(v,r_0/2^{k-1})\bigg\|_{L^\infty(B_1)}\\
					&\leq C \|D^2 v_{u(y)}\|_{L^{\infty}(B_1)}+C \sum\limits_{k=1}^{\infty} \omega\left(\frac{cr}{2^{k-1}}\right)\le \tilde C(\|D^2 v_{u(y)}\|_{L^{\infty}(B_1)}+1)\\
					&\leq \tilde C\left(\sup_{|s| \leq \sup |u|} \|D^2 v_s\|_{L^\infty(B_1)}+1\right).
					\end{align*}
					Thus 
					
					\begin{align}\label{eq: 1}
					\|\Pi_y(u,r_0/2^j)\|_{L^\infty (B_1)} &\le \|\Pi_y(u,r_0)\|_{L^\infty(B_1)} + \tilde{C}(\|D^2 v_{u(y)}\|_{L^{\infty}}+1).
					\end{align}
					We conclude via Lemma \ref{c11} and Lemma \ref{log}.
				\end{proof}
				
				\begin{rem}
					To generate examples, consider $f(x,t)=\phi(x) \psi(t)$. If $\phi \in L^\infty$ and $\psi$ is Dini, then $f$ satisfies condition (i). If $\phi * N$ is $C^{1,1}$ and $\psi$ is locally bounded, then $f$ satisfies (ii). Thus if $\phi * N$ is $C^{1,1}$ and $\psi$ is Dini, then $f$ satisfies both conditions. In particular, $f$ may be strictly weaker than Dini in the $x$-variable.
				\end{rem}
				\begin{rem}\label{rmrk: ineq_for_Q}
					The projection $Q_y$ has similar properties to $\Pi_y$. Consequently, if $f$ satisfies assumption \textbf{A}, \eqref{eq: 1} holds for $\Pi_y$ replaced by $Q_y$.  
				\end{rem}

				\section{$C^{1,1}$ regularity: discontinuous case}
				
				The goal of this section is to investigate the optimal regularity for solutions of \eqref{pro} with $f$ having a jump discontinuity in the $t$-variable. This case may be viewed as a free boundary problem. The idea is to employ again an $L^2$ projection operator.

									\subsection{Two-phase obstacle problem }
									Suppose $f=f(x,u)$ has the form   
									
									\begin{equation}\label{f form}
										f(x,u)=g_1(x,u)\chi_{\{u>0\}}+g_2(x,u)\chi_{\{u<0\}},
									\end{equation}
									where $g_1,g_2$ are continuous. We recall from the introduction that if $f$ has a jump in $u$ at the origin, then we assume it to be a positive jump:   
									\vskip .3 cm
									\textbf{Assumption B.} $g_1(x,0)-g_2(x,0) \geq \sigma_0,\ x \in B_1$ for some $\sigma_0>0$.
									\vskip .3 cm
					
					\begin{rem}
In the unstable obstacle problem, i.e. $g_1=-1$, $g_2=0$, there exists a solution which is $C^{1,\alpha}$ for any $\alpha \in (0,1)$ but not $C^{1,1}$.  
\end{rem}
					
					Let $\Gamma^0:=\Gamma \cap \{|\nabla u|=u=0\}$ and $\Gamma^1:=\Gamma \cap \{|\nabla u| \neq 0\}$. Our main result provides optimal growth away from points with sufficiently small gradients.
									
\begin{thm}\label{maintheorem}
	Suppose $g_1,g_2 \in C^0$ satisfy \textbf{B}. Then for all constants $\theta,M>0$ there exist $r_0(\theta,M,\|g_1\|_{\infty},\|g_2\|_{\infty},n)>0$ and $C_0(\theta,M,\|g_1\|_{\infty},\|g_2\|_{\infty},n)>0$ such that for any solution of \eqref{pro} with $\|u\|_{L^\infty(B_1)} \leq M$
	\begin{equation}\label{eq: C_0_ineq}
		\|Q_y(u,r)\|_{L^2(\partial B_1(0))} \leq C_0,
	\end{equation}
	for all $r\leq r_0$ and $y \in B_{1/2} \cap \Gamma \cap\{|\nabla u(y)| <\theta r\}$. Consequently, for the same choice of $r$ and $y$ we have that
	\begin{equation}\label{eq: C_1_ineq}
		\sup\limits_{x \in B_r}|u(x+y)-x \cdot \nabla u(y)|\leq C_1 r^2,
	\end{equation}
	 for some constant $C_1(\theta,M,\|g_1\|_{\infty},\|g_2\|_{\infty},n)>0$.
\end{thm}

The proof of the theorem is carried out in several steps. A crucial ingredient is the following monotonicity result.

\begin{lem}\label{normdecay}
	Suppose $g_1,g_2 \in C^0$ satisfy \textbf{B}. Then for all constants $\theta,M>0$ there exist $\kappa_0(\theta, M,\|g_1\|_{\infty},\|g_2\|_{\infty},n)>0$ and $r_0(\theta,M,\|g_1\|_{\infty},\|g_2\|_{\infty},n)>0$ such that for any solution $u$ of \eqref{pro} with $\|u\|_{L^\infty(B_1)} \leq M$ if
	\[\|Q_y(u,r)\|_{L^2(\partial B_1)} \geq \kappa_0,
	\]
	for some $0<r<r_0$ and $y \in B_{1/2} \cap \Gamma \cap\{|\nabla u(y)| <\theta r\}$, then
	\begin{equation}
		\frac{d}{dr} \int\limits_{\partial B_1} Q_y^2(u,r) d \HH^{n-1}>0.
	\end{equation}
\end{lem}
\begin{proof}
	If the conclusion is not true, then there exist radii $r_k \to 0$, solutions $u_k$ and points $y_k \in B_{1/2} \cap \Gamma_k \cap\{|\nabla u_k(y_k)| <\theta r_k\}$ such that $\|u_k\|_{L^\infty(B_1)} \le M$, and $\|Q_{y_k}(u_k,r_k)\|_{L^2(\partial B_1)} \to \infty$, and
	\[\frac{d}{dr} \int\limits_{\partial B_1} Q_{y_k}^2(u_k,r) d \HH^{n-1}\bigg|_{r=r_k}\leq 0.\]
	Let
	\[T_k:=\|Q_{y_k}(u_k,r_k)\|_{L^2(\partial B_1)},\]
	and consider the sequence
	\[v_k(x)=\frac{u_k(r_kx+y_k)-r_kx\cdot \nabla u_k(y_k)}{r_k^2}-Q_{y_k}(u_k,r_k).\]
	Without loss of generality we can assume that $y_k\to y_0$ for some $y_0 \in B_{1/2}$. Lemma \ref{w2p} implies the existence of a function $v$ such that up to a subsequence
	\[v_k(x)=\frac{u_k(r_kx+y_k)-r_kx\cdot \nabla u_k(y_k)}{r_k^2}-Q_{y_k}(u_k,r_k) \rightarrow v,\ \text{in}\ C_{\text{loc}}^{1,\alpha}(\RR^n)\cap W^{2,p}_{\text{loc}}(\RR^n).
	\]
	Evidently, $v(y_0)=|\nabla v(y_0)|=0$. Moreover, for $q_k(x):=Q_{y_k}(u_k,r_k)/T_k$, we can assume that up to a further subsequence, $q_k \to q$ in $C^{\infty}$ for some $q \in \Pp_2$. Note that 
	\begin{align}
		\Delta v_k(x)&=g_1(r_kx+y_k,u_k(r_k x+y_k))\chi_{\{u_k(r_kx+y_k)>0\}}\\
		&+g_2(r_kx+y_k,u_k(r_kx+y_k))\chi_{\{u_k(r_kx+y_k)<0\}}
	\end{align}
	hence
	\begin{equation}
	\Delta v_k \to \Delta v = g_1(y_0,0)\chi_{\{q(x)>0\}}+g_2(y_0,0)\chi_{\{q(x)<0\}}.
	\end{equation}
	By Lemma \ref{normderivative}, 
	\begin{align}
		0&\geq\frac{d}{dr} \int\limits_{\partial B_1} Q_{y_k}^2(u_k,r) d \HH^{n-1}\bigg|_{r=r_k}=\frac{2}{r_k}\int\limits_{B_1} Q_{y_k}(u_k,r_k) \Delta u_k(r_kx+y_k) d x\\
		&=\frac{2T_k}{r_k}\int\limits_{B_1} q_k(x) \Delta v_k(x) d x.
	\end{align}
	Therefore $$\int\limits_{B_1} q_k(x) \Delta v_k(x) d x \leq 0.$$ On the other hand
	\begin{align}
	\lim\limits_{k \to \infty}& \int\limits_{B_1} q_k(x) \Delta v_k(x) d x=\int\limits_{B_1} q(x)\left(g_1(0,y_0)\chi_{\{q(x)>0\}}+g_2(0,y_0)\chi_{\{q(x)<0\}}\right) dx\\
	&=\left(g_1(0,y_0)-g_2(0,y_0)\right)\int\limits_{q(x)>0} q(x)dx>0,
	\end{align}
	a contradiction.
\end{proof}
	
\begin{proof}[\textbf{Proof of Theorem \ref{maintheorem}}]
	Let $\kappa_0$ and $r_0$ be the constants from Lemma \ref{normdecay}. Without loss of generality we can assume that $r_0\leq 1/4$. From Lemmas \ref{log} and \ref{u-v} we have that
	\[\|Q_y(u,r_0)\|_{L^2(\partial B_1)}\leq C \log\frac{1}{r_0},
	\]
	for all $y \in B_{1/2}$, where $C=C(M,\|g_1\|_{\infty},\|g_2\|_{\infty},n)$ is a constant. Take
	\begin{equation}\label{eq: M_0}
		C_0=\max \left(k_0,2C\log \frac{1}{r_0}\right).
	\end{equation}
	We claim that
	\[\|Q_y(u,r)\|_{L^2(\partial B_1)} \leq C_0,
	\]
	for $r\leq r_0$ and $y \in B_{1/2} \cap \Gamma \cap\{|\nabla u(y)| <\theta r\}$. Let us fix $y$ such that $|\nabla u(y)|\leq \theta r_0$ and consider
	\begin{equation}\label{eq: T_y}
		T_y(r):=\|Q_y(u,r)\|_{L^2(\partial B_1)}
	\end{equation}
	as a function of $r$ on the interval $\nabla |u(y)|/\theta \leq r\leq r_0$. Let
	\begin{equation}\label{eq: e}
		e:=\inf \{r\ \text{s.t.}\ T_y(r)\leq C_0\}.
	\end{equation}
	We have that $T_y(r_0)<C_0$, so $|\nabla u(y)|/\theta\leq e<r_0$. If $e>|\nabla u(y)|/\theta$ then $T_y(e)=C_0$ and by Lemma \ref{normdecay} we have that $T_{y}'(e)>0$, so $T_y(r)<C_0$ for $e-\eps<r<e$ which contradicts \eqref{eq: e}.
	
	Therefore, $e=|\nabla u(y)|/\theta$ and $T_y(r)\leq C_0$ for all $|\nabla u(y)|/\theta\leq r\leq r_0$ which proves \eqref{eq: C_0_ineq}.
	
	Inequality \eqref{eq: C_1_ineq} follows from Lemmas \ref{w2p} and \ref{u-v}. 
	\end{proof}

Theorem \ref{maintheorem} implies $C^{1,1}$ regularity away from $\Gamma^1$ in the case the coefficients $g_i$ are regular enough to provide $C^{1,1}$ solutions away from the free boundary, i.e. Theorem \ref{AB}.
%
%
\begin{rem}

	Note that \textbf{A} is the condition given in Theorem \ref{contcase}. If $g_i$ only depend on $x$, then this reduces to the assumption that the Newtonian potential of $g_i$ is $C^{1,1}$, which is sharp. 
\end{rem}
\begin{proof}[\bf{Proof of Theorem \ref{AB}}]

		Suppose \textbf{A} and \textbf{B} hold. We show that for every $\delta>0$ there exists $C_{\delta}>0$ such that for all $y \in B_{1/2}(0)$ such that $\dist (y,\Gamma^1) \geq \delta$, there exists $r_y>0$ such that
	\begin{equation}\label{eq: Q_estimate}
		\|Q_y(u,r)\|_{L^2(\partial B_1(0))}\leq C_{\delta},
	\end{equation}
	for $r\leq r_y$.
	
	Consequently,
	\begin{equation}\label{eq: Taylor_estimate}
		|u(x)-u(y)-\nabla u(y) (x-y)| \leq \tilde{C}_{\delta}|x-y|^2
	\end{equation}
	for $|x-y|\leq r_y,\ y\in B_{1/2}(0)$ and $\dist (y,\Gamma^1) \geq \delta$; this readily yields the desired result.
	

	Note that \eqref{eq: Taylor_estimate} follows from \eqref{eq: Q_estimate} via Lemmas \ref{w2p} and \ref{u-v}. 
	
	Without loss of generality assume that $\delta\leq r_0$, where $r_0>0$ is the constant from Theorem \ref{maintheorem}. For every $y \in B_{1/2}(0)$ consider the ball $B_{\delta/2}(y)$. Then there are two possibilities.
	\begin{itemize}
		\item[i.] $B_{\delta/2}(y) \cap \Gamma^0=\emptyset$.
		
		In this case $B_{\delta/2} \cap \Gamma=\emptyset$, hence $u$ satisfies the equation
		\[\Delta u = g_i(x,u)
		\]
		in $B_{\delta/2}(y)$ for $i=1$ or $i=2$. Inequality \eqref{eq: 1} in the Theorem \ref{contcase} assumption \textbf{A} yields
		\begin{equation}
			\|Q_y(u,r)\|_{L^2(\partial B_1(0))}\leq C\log \frac{4}{\delta}+C(\|D^2v^i_{u(y)}\|_{\infty}+1),
		\end{equation}
		for $r\leq \delta/4$.
		\item[ii.] $B_{\delta/2}(y) \cap \Gamma^0 \neq \emptyset$.
		
		Let $w\in \Gamma^0$ be such that $d:=|y-w|=\dist (y,\Gamma_0)$. We have that $d\leq \delta/2$. As before, assumption \textbf{A} yields
		\begin{equation}
		\|Q_y(u,r)\|_{L^2(\partial B_1(0))}\leq \|Q_y(u,d/2)\|_{L^2(\partial B_1(0))}+C(\|D^2v^i_{u(y)}\|_{\infty}+1),
		\end{equation}
		for $r\leq d/2$. From Theorem \ref{maintheorem} we have that
		\begin{equation}
			\left|u\left(y+\frac{d}{2}z\right)\right| \leq C\left|y+\frac{d}{2}z-w\right|^2\leq C d^2,
		\end{equation}
		 for all $|z|\leq 1$ because $d\leq \delta/2\leq r_0$. On the other hand
		\begin{align}
			Q_y(u,d/2)&=\Proj _{\Pp_2}\left(\frac{u\left(y+\frac{d}{2}z\right)-\frac{d}{2}z\cdot \nabla u(y)-u(y)}{d^2/4}\right)\\
			&=\Proj _{\Pp_2}\left(\frac{u\left(y+\frac{d}{2}z\right)}{d^2/4}\right),
		\end{align}
		where $\Proj _{\Pp_2}$ is the $L^2(\partial B_1(0))$ projection on the space $\Pp_2$. We have used the fact that the projection of a linear function is 0. Hence
		\begin{equation}
			\|Q_y(u,d/2)\|_{L^2(\partial B_1(0))}\leq \left\|\frac{u\left(y+\frac{d}{2}z\right)}{d^2/4}\right\|_{L^2(\partial B_1(0))}\leq C,
		\end{equation}
		which yields
		\begin{equation}
		\|Q_y(u,r)\|_{L^2(\partial B_1(0))}\leq C+C(\|D^2v^i_{u(y)}\|_{\infty}+1),
		\end{equation}
		for $r\leq d/2$.
	\end{itemize}
	The proof is now complete.
\end{proof}

Lastly we point out that if the coefficients $g_i$ are regular enough to provide $C^{1,1}$ solutions at points where the gradient does not vanish, then we obtain full interior $C^{1,1}$ regularity.

\vskip .3 cm
\textbf{Assumption C.} For any $M>0$ there exist $\theta_0(M,\|g_1\|_{\infty},\|g_2\|_{\infty},n)>0$ and $C_3(M,\|g_1\|_{\infty},\|g_2\|_{\infty},n)>0$ such that for all $z\in B_{1/2}$ any solution of
\begin{equation}
\begin{cases}
\Delta v = g_1(x,v)\chi_{v>0}+g_2(x,v)\chi_{v<0},\ x \in B_{1/2}(z);\\
|v(x)|\leq M,\ x \in B_{1/2}(z);\\
v(z)=0,\ 0<|\nabla v(z)|\leq \theta_0 /4;\\
v\big|_{\partial B_r(z)} \ \text{continuous},
\end{cases}
\end{equation}
admits a bound
\begin{equation}
\|D^2 v\|_{L^{\infty}(B_{|\nabla v(z)|/\theta_0}(z))} \leq C_3.
\end{equation}

\vskip .3 cm

\begin{rem}
A sufficient condition which ensures \textbf{C} is that $g_i$ are H\"older continuous, see \cite[Proposition 2.6]{LSE09} and \cite[Theorem 9.3]{ADN64}. The idea being that at such points, the set $\{u=0\}$ is locally $C^{1,\alpha}$ (via the implicit function theorem) and one may thereby reduce the problem to a classical PDE for which up to the boundary estimates are known. 
\end{rem}

Theorem \ref{maintheorem} and \textbf{C} imply Theorem \ref{ABC}.

\begin{proof}[\textbf{Proof of Theorem \ref{ABC}}]
By Lemmas \ref{u-v} and \ref{c11} the assertion follows if we show that there exist $\rho_0, C>0$ such that for every $y \in B_{\rho_0}(0)$ there exists $r_y>0$ such that
\begin{equation}\label{eq: goal}
	\|Q_y(u,r)\|_{L^2(\partial B_1(0))} \leq C
\end{equation}
       	
for $0<r\leq r_y$.

Let $\rho_0$ be such that $|\nabla u(y)| \leq \theta_0$ for $y \in B_{\rho_0}(0)$, where $\theta_0$ is the constant from assumption \textbf{C} (we can do this because $u$ is $C^{1,\alpha}$ and $0 \in\Gamma^0$). For $y \in B_{\rho_0}(0)$ let $d:=\dist (y,\Gamma)$ and let $w \in \Gamma$ be such that $d=|y-w|$.

 From Corollary \ref{AB} we can assume that $2d<r_0$. One of the following cases is possible.
\begin{itemize}
	\item[i.] $d=0, y \in \Gamma^0$.
	
	In this case we have that \eqref{eq: goal} holds for $r\leq r_0$ by Theorem \ref{maintheorem}.
	
	\item[ii.] $d=0, y \in \Gamma^1$.
	
	Here, \eqref{eq: goal} follows from the assumption $C$.
	
	\item[iii.] $d>0, w \in \Gamma^0$.
	
	$u$ solves $\Delta u=g_i(x,u)$ in $B_{d/2}(y)$ for $i=1$ or $i=2$. Then, by the analysis similar to the one in Corollary \ref{AB} we get that \eqref{eq: goal} holds for $r\leq d/2$.
	
	\item[iv.] $d>0, w \in \Gamma^1$.
	
	From Theorem \ref{maintheorem} we have that
	\begin{equation}\label{eq: ineq}
		|u(z+w)-z\cdot \nabla u(w)| \leq C_1 |z|^2
	\end{equation}
	for $|\nabla u(w)|/\theta_0 \leq |z|\leq r_0$. On the other hand by assumption \textbf{C} we obtain that \eqref{eq: ineq} holds for $|z| \leq |\nabla u(w)|/\theta_0$. Hence, \eqref{eq: ineq} holds for all $z$ such that $|z|\leq r_0$.
	
	By assumption \textbf{A} we have that
	\begin{equation}
		\|Q_y(u,r)\|_{L^2(\partial B_1(0))}\leq \|Q_y(u,d/2)\|_{L^2(\partial B_1(0))}+C(\|D^2v^i_{u(y)}\|_{\infty}+1),
	\end{equation}
	for $r\leq d/2$.
	
	Furthermore,
	\begin{align}
	Q_y(u,d/2)&=\Proj _{\Pp_2}\left(\frac{u\left(y+\frac{d}{2}z\right)-\frac{d}{2}z\cdot \nabla u(y)-u(y)}{d^2/4}\right)\\
	&=\Proj _{\Pp_2}\left(\frac{u\left(y+\frac{d}{2}z\right)-\left(y+\frac{d}{2}z-w\right)\cdot \nabla u(w)}{d^2/4}\right).\\
	\end{align}
	Hence from \eqref{eq: ineq} we get
	\begin{align}
	\|Q_y(u,d/2)\|_{L^2(\partial B_1(0))}&\leq \left\|\frac{u\left(y+\frac{d}{2}z\right)-\left(y+\frac{d}{2}z-w\right)\cdot \nabla u(w)}{d^2/4}\right\|_{L^2(\partial B_1(0))}\\
	&\leq C,
	\end{align}
	which yields
	\begin{equation}
	\|Q_y(u,r)\|_{L^2(\partial B_1(0))}\leq C+C(\|D^2v^i_{u(y)}\|_{\infty}+1),
	\end{equation}
	for $r\leq d/2$.

\end{itemize} 

\end{proof}
The previous analysis applies to the following example. 

\vskip .2in

\textbf{Example.} Let $g_i(x,u)=\lambda_i(x)$ for $i=1,2$, where $\lambda_i$ are such that
\begin{itemize}
	\item[i.] $\lambda_1(x)-\lambda_2(x)\geq \sigma_0>0$ for all $x \in B_1$;
	\item[ii.] $\lambda_1(x),\lambda_2(x)$ are H\"{o}lder continuous.
	\end{itemize}
We recall from the introduction that under the stronger assumption $\inf_{B_1} \lambda_1>0$, $\inf_{B_1}-\lambda_2>0$, this problem is studied in \cite{LSE09} and the optimal interior $C^{1,1}$ regularity is established. The authors use a different approach based on monotonicity formulas and an analysis of global solutions via a blow-up procedure. 


\subsection{No-sign obstacle problem}

Here we observe that assumption \textbf{A} implies that the solutions of \eqref{pro: no-sign} are in $C^{1,1}(B_{1/2})$.
This theorem was proven in \cite{ALS13} (Theorem 1.2) for the case when $g(x,t)$ depends only on $x$. Under assumption \textbf{A},  appropriate modifications of the proof in \cite{ALS13} work also for the general case; since the arguments are similar, we provide only a sketch of the proof and highlight the differences. 

\begin{proof}[\textbf{Sketch of the proof of Theorem \ref{thm: no-sign}}] Let $\tilde{\Gamma}:=\{y\ \text{s.t.}\ u(y)=|\nabla u(y)|=0\}$. For $r>0$ let $\Lambda_r:=\{x \in B_1\ \text{s.t.}\ u(rx)=0 \}$ and $\lambda_r:=|\Lambda_r|$.
	
	The proof of Theorem 1.2 in \cite{ALS13} consists of the following ingredients.
	\begin{itemize}
		\item Interior $C^{1,1}$ estimate
		\item Quadratic growth away from the free boundary
		\item  \cite[Proposition 5.1]{ALS13}
	\end{itemize}
	Let us recall that the interior $C^{1,1}$ estimate is the inequality
	\begin{equation}\label{eq: int_estimate}
	\|u\|_{C^{1,1}(B_{d/2})} \leq C \left(\|g\|_{L^{\infty}(B_d)}+\frac{\|u\|_{L^{\infty}(B_d)}}{d^2}\right),
	\end{equation}
	where $\Delta u (x)= g(x)$ for $x \in B_d$ and the Newtonian potential of $g$ is $C^{1,1}$. This estimate is purely a consequence of $g$ having a $C^{1,1}$ Newtonian potential.
	
	Quadratic growth away from the free boundary is a bound
	\begin{equation}\label{eq: quadgrowth-no-sign}
	|u(x)| \leq C \dist (x,\tilde{\Gamma})^2.
	\end{equation}
	The first observation in \cite{ALS13} is that if $g(x,t)=g(x)$ has a $C^{1,1}$ Newtonian potential, then \eqref{eq: quadgrowth-no-sign} and \eqref{eq: int_estimate} yield $C^{1,1}$ regularity for the solution. Indeed, "far" from the free boundary, the solution $u$ solves the equation $\Delta u=g(x)$ and is locally $C^{1,1}$ by assumption. For points close to the free boundary, $u$ solves the same equation but now on a small ball centered at the point of interest and touching the free boundary. At this point one invokes \eqref{eq: quadgrowth-no-sign} and by \eqref{eq: int_estimate} obtains that the $C^{1,1}$ bound does not blow up close to the free boundary (see Lemma 4.1 in \cite{ALS13}).
	
	To prove \eqref{eq: quadgrowth-no-sign}, the authors prove in Proposition 5.1 \cite{ALS13} that if the projection $\Pi_y(u,r)$ (for some $y \in \tilde{\Gamma}$) is large enough then the density $\lambda_r$ of the coincidence set diminishes at an exponential rate. On the other hand, if $\lambda_r$ diminishes in an exponential rate, $\Pi_y(u,r)$ has to be bounded. Consequently, by invoking Lemma \ref{w2p} one obtains \eqref{eq: quadgrowth-no-sign}.
	
	Now let $g$ satisfy \textbf{A}.
	\begin{itemize}
		\item Interior $C^{1,1}$ estimate
		
		In the general case, \eqref{eq: int_estimate} is replaced by
		\begin{equation}\label{eq: int_estimate_general}
		\|Q_y(u,s)\|_{L^2(\partial B_1(0))}\leq \|Q_y(u,r)\|_{L^2(\partial B_1(0))}+C(\|D^2v_{u(y)}\|_{\infty}+1),
		\end{equation}
		where $0<s<r<d,\ \Delta v_{u(y)}=g(x,u(y))$ and $\Delta u=f(x,u)$ in $B_d(y)$. Estimate \eqref{eq: int_estimate_general} is purely a consequence of assumption \textbf{A} (see \eqref{eq: 1} in the proof of Theorem \ref{contcase}).
		
		\item \cite[Proposition 5.1]{ALS13}
		
		In this proposition, it is shown that there exists $C$ such that if $\Pi_y(u,r)\geq C$ then
		\begin{equation}\label{eq: exp_decay_ALS}
		\lambda_{r/2}^{1/2}\leq \frac{\tilde{C}}{\|\Pi_y(u,r)\|_{L^\infty(B_1)}}\lambda_{r}^{1/2}
		\end{equation}
		for some $\tilde{C}>0$. The inequality is obtained by the decomposition
		\[\frac{u(rx+y)}{r^2}=\Pi_y(u,r)+h_r+w_r,
		\]
		where $h_r,w_r$ are such that
		\begin{equation}
		\begin{cases}
		\Delta h_r=-g(rx+y)\chi_{\Lambda_r} & \text{in}\ B_1,\\
		h_r=0 &\text{on}\ \partial B_1,
		\end{cases}
		\end{equation}
		and
		\begin{equation}
		\begin{cases}
		\Delta w_r=g(rx+y) & \text{in}\ B_1,\\
		w_r= \frac{u(rx+y)}{r^2}-\Pi_y(u,r)&\text{on}\ \partial B_1.
		\end{cases}
		\end{equation}
		The authors show that
		\begin{align}\label{eq: 5.1ineq}
		\|D^2h_r\|_{L^2(B_{1/2})}&\leq C \|g\|_{L^\infty}\|\chi_{\Lambda_r}\|_{L^2(B_1)},\\ \nonumber
		\|D^2w_r\|_{L^{\infty}(B_{1/2})}&\leq C \left(\|g\|_{L^\infty}+\|u\|_{L^{\infty}(B_1)}\right).
		\end{align}
		In the general case one may consider the decomposition
		\[\frac{u(rx+y)}{r^2}=Q_y(u,r)+h_r+w_r+z_r,
		\]
		where $h_r,w_r,z_r$ are such that
		\begin{equation}
		\begin{cases}
		\Delta h_r=-g(rx+y,0)\chi_{\Lambda_r} & \text{in}\ B_1,\\
		h_r=0 &\text{on}\ \partial B_1,
		\end{cases}
		\end{equation}
		and
		\begin{equation}
		\begin{cases}
		\Delta w_r=g(rx+y,0) & \text{in}\ B_1,\\
		w_r= \frac{u(rx+y)}{r^2}-Q_y(u,r)&\text{on}\ \partial B_1,
		\end{cases}
		\end{equation}
		and
		\begin{equation}
		\begin{cases}
		\Delta z_r=\left(g(rx+y,u(rx+y))-g(rx+y,0)\right)\chi_{B_1\setminus \Lambda_r} & \text{in}\ B_1,\\
		z_r= 0&\text{on}\ \partial B_1.
		\end{cases}
		\end{equation}
		Evidently, estimates \eqref{eq: 5.1ineq} are still valid. Additionally, we have 
		\begin{align}\label{eq: 5.1ineq_general}
		\|D^2z_r\|_{L^2(B_{1/2})}\leq C\|\Delta z_r\|_{L^2(B_1)}\leq  C \omega(r^2 \log \frac{1}{r}),
		\end{align}
		since $g(x,t)$ is uniformly Dini in $t$.
		
		Combining \eqref{eq: 5.1ineq} and \eqref{eq: 5.1ineq_general} and arguing as in \cite{ALS13} one obtains the existence of $C>0$ such that
		\begin{align}\label{eq: exp_decay}
		\lambda_{r/2}^{1/2}\leq \frac{\tilde{C}}{\|Q_y(u,r)\|_{L^2(\partial B_1)}}\lambda_{r}^{1/2}+\omega\left(r^2\log \frac{1}{r}\right),
		\end{align}
		whenever $\|Q_y(u,r)\|_{L^2(\partial B_{1})} \geq C$.
		\item Quadratic growth away from the free boundary
		
		In \cite{ALS13}, the norms of $\Pi_y(u,r/2^k),\ k\geq 1$ are estimated in terms of the sum $\sum\limits_{j=0}^{\infty} \lambda_{r/2^j}$.  If the norms of projections are unbounded, one obtain estimate \eqref{eq: exp_decay_ALS} which implies convergence of the previous sum and hence boundedness of the projections. This is a contradiction.
		
		Similarly, in the general case the norms of $Q_y(u,r/2^k),\ k\geq 1$ can be estimated by \[\sum\limits_{j=0}^{\infty} \lambda_{r/2^j}+\sum\limits_{j=0}^{\infty}\omega \left(\left(\frac{r}{2^k}\right)^2 \log\frac{2^k}{r^2} \right).\]
		Inequality \eqref{eq: exp_decay} and Dini continuity imply 
		\[\sum\limits_{j=0}^{\infty}\omega \left(\left(\frac{r}{2^k}\right)^2 \log\frac{2^k}{r^2} \right)\ ,\ \sum\limits_{j=0}^{\infty} \lambda_{r/2^j}<\infty,\]
		if the norms of projections are unbounded. Furthermore, one completes the proof of the quadratic growth as in \cite{ALS13}.
		
		To verify that the above ingredients imply $C^{1,1}$ regularity, we split the analysis into two cases. If we are "far" from the free boundary, $u$ locally solves $\Delta u=g(x,u)$ so by Theorem 3.1 $u$ is $C^{1,1}$. If we are close to the free boundary then $u$ solves $\Delta u=g(x,u)$ in a small ball $B_d(y)$ that touches the free boundary. We invoke \eqref{eq: int_estimate_general} for $0<s<r=d/2$ and the quadratic growth to obtain
		\begin{align}
		\|Q_y(u,s)\|_{L^2(\partial B_1(0))}&\leq \|Q_y(u,d/2)\|_{L^2(\partial B_1)}+C(\|D^2v_{u(y)}\|_{\infty}+1)\\
		&\leq C\bigg\|\frac{u(y+d/2 x)}{d^2/4}\bigg\|_{L^2(\partial B_1)}+C(\|D^2v_{u(y)}\|_{\infty}+1)\\
		&\leq C+C(\|D^2v_{u(y)}\|_{\infty}+1).
		\end{align}
		for $s\leq d/2$.
		
		So there exists a constant $C$ such that for all $y \in B_{1/2}$ there exist radii $r_j(y) \to 0$ such that
		\[Q_y(u,r_j(y))\leq C.
		\]
		We conclude via Lemma \ref{c11}.
	\end{itemize}
	\end{proof}
	
	\paragraph{Acknowledgements} We thank Henrik Shahgholian for introducing us to the regularity problem for semilinear equations. Special thanks go to John Andersson for valuable feedback on a preliminary version of the paper.  E. Indrei acknowledges partial support from NSF Grants OISE-0967140 (PIRE), DMS-0405343, and DMS-0635983 administered by the Center for Nonlinear Analysis at Carnegie Mellon University and an AMS-Simons Travel Grant. L. Nurbekyan was partially supported by KAUST baseline and start-up funds and KAUST SRI, Uncertainty Quantification Center in Computational Science and Engineering.

\bibliographystyle{amsalpha}
\bibliography{References}
				
\pagebreak
				
\signei
				
\signam
				
\signln
				
\end{document}